\documentclass[11pt]{amsart}

\issueinfo{00}{0}{Month}{Year}
\copyrightinfo{2008}{University of Calgary}
\pagespan{1}{2}
\PII{ISSN 1715-0868}

\usepackage{color}      
\usepackage{epsfig}
\usepackage{graphicx}           

\newtheorem{theorem}{Theorem}[section]
\newtheorem{corollary}[theorem]{Corollary}
\newtheorem{proposition}[theorem]{Proposition}
\newtheorem{observation}[theorem]{Observation}

\begin{document}

\title{Domination Value in Graphs}

\author{Eunjeong Yi}
\address{Texas A\&M University at Galveston, Galveston, TX 77553, USA}
\email{yie@tamug.edu}

\date{(date1), and in revised form (date2).}


\thanks{}

\begin{abstract}
A set $D \subseteq V(G)$ is a \emph{dominating set} of $G$ if
every vertex not in $D$ is adjacent to at least one vertex in $D$.
A dominating set of $G$ of minimum cardinality is called a
$\gamma(G)$-set. For each vertex $v \in V(G)$, we define the
\emph{domination value} of $v$ to be the number of
$\gamma(G)$-sets to which $v$ belongs. In this paper, we study
some basic properties of the domination value function, thus
initiating \emph{a local study of domination} in graphs. Further,
we characterize domination value for the Petersen graph, complete
$n$-partite graphs, cycles, and paths.
\end{abstract}

\maketitle

\section{Introduction}

Let $G = (V(G),E(G))$ be a simple, undirected, and nontrivial
graph with order $|V(G)|$ and size $|E(G)|$. For $S \subseteq
V(G)$, we denote by $\langle S \rangle$ the subgraph of $G$ induced by $S$.
The \emph{degree of a vertex $v$} in $G$, denoted by $\deg_G(v)$,
is the number of edges that are incident to $v$ in $G$; an \emph{end-vertex} is a vertex of degree one, 
and a \emph{support vertex} is a vertex that is adjacent to an end-vertex. We denote by
$\Delta(G)$ \emph{the maximum degree} of a graph $G$. For a vertex
$v \in V(G)$, the \emph{open neighborhood} $N(v)$ of $v$ is the
set of all vertices adjacent to $v$ in $G$, and the \emph{closed
neighborhood}  $N[v]$ of $v$ is the set $N(v) \cup \{v\}$. A set $D
\subseteq V(G)$ is a \emph{dominating set} (DS) of $G$ if for each
$v \not\in D$ there exists a $u \in D$ such that $uv\in E(G)$. The
\emph{domination number} of $G$, denoted by $\gamma(G)$, is the
minimum cardinality of a DS in $G$; a DS of $G$ of minimum
cardinality is called a $\gamma(G)$-set. For earlier discussions
on domination in graphs, see \cite{B1, B2, EJ, Jaegar, Ore}. For a
survey of domination in graphs, refer to \cite{Dom1, Dom2}. We generally
follow \cite{CZ} for notation and graph theory terminology.
Throughout the paper, we denote by $P_n$, $C_n$, and $K_n$ the path,
the cycle, and the complete graph on $n$ vertices, respectively.\\

In \cite{Slater}, Slater introduced the notion of the number of dominating sets of $G$, which he denoted by HED$(G)$ in honor of Steve Hedetniemi; further, he also 
used \#$\gamma(G)$ to denote the number of $\gamma(G)$-sets. In this paper, we will use $\tau(G)$ to denote the total number of $\gamma(G)$-sets, and
by $DM(G)$ the collection of all $\gamma(G)$-sets. For each vertex
$v \in V(G)$, we define the \emph{domination value} of $v$,
denoted by $DV_G(v)$, to be the number of $\gamma(G)$-sets to
which $v$ belongs; we often drop $G$ when ambiguity is not a
concern. See \cite{CK} for a discussion on \emph{total domination value} in graphs. For a further work on domination value in graphs, see \cite{Yi}. In this paper, we study some basic
properties of the domination value function, thus initiating a
\emph{local study} of domination in graphs. When a real-world
situation can be modeled by a graph, the locations (vertices)
with high domination values are of interest. One can use
domination value in selecting locations for fire departments or
convenience stores, for example. Though numerous papers on
domination have been published, no prior systematic local study of
domination is known. However, in \cite{pruned tree}, Mynhardt
characterized the vertices in a tree $T$ whose domination value is $0$ or $\tau(T)$. 
It should be noted that finding domination value of any given vertex in 
a given graph $G$ can be an extremely difficulty task, given the difficulty attendant to finding $\tau(G)$ or just $\gamma(G)$.\\


\section{Basic properties of domination value: upper and lower bounds}

In this section, we consider the lower and upper bounds of the
domination value function for a fixed vertex $v_0$ and for $v \in
N[v_0]$. Clearly, $0 \le DV_G(v) \le \tau(G)$ for any graph $G$
and for any vertex $v \in V(G)$. We will say the bound is sharp if
equality is obtained for a graph of some order in an inequality.
We first make the following observations.\\

\begin{observation}\label{observation}
$\displaystyle \sum_{v \in V(G)} DV_G(v) = \tau(G) \cdot
\gamma(G)$
\end{observation}

\begin{observation}\label{observation1}
If there is an isomorphism of graphs carrying a vertex $v$ in $G$
to a vertex $v'$ in $G'$, then $DV_G(v)=DV_{G'}(v')$.
\end{observation}

Examples of graphs that admit automorphisms are cycles, paths, and
the Petersen graph. The Pertersen graph, which is often used as a
counter-example for conjectures, is vertex-transitive (p.27,
\cite{Petersen}). Let $\mathcal{P}$ denote the Petersen graph with
labeling as in Figure \ref{P}.

\begin{figure}[htbp]
\begin{center}
\scalebox{0.5}{\input{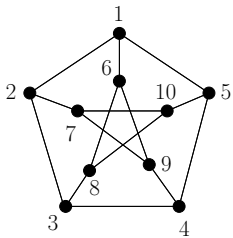}} \caption{The Petersen
graph}\label{P}
\end{center}
\end{figure}

It's easy to check that $\gamma(\mathcal{P})=3$. We will show that
$DV(v)=3$ for each $v \in V(\mathcal{P})$. Since $\mathcal{P}$ is
vertex-transitive, it suffices to compute $DV_{\mathcal{P}}(1)$.
For the $\gamma(\mathcal{P})$-set $\Gamma$ containing the vertex
$1$, one can easily check that no vertex in $N(1)$ belongs to
$\Gamma$. Further, notice that no three vertices from
$\{1,2,3,4,5\}$ form a $\gamma(\mathcal{P})$-set. Keeping these
two conditions in mind, one can readily verify that the
$\gamma(\mathcal{P})$-sets containing the vertex $1$ are
$\{1,3,7\}$, $\{1,4,10\}$, and $\{1,8,9\}$, and thus
$DV(1)=3=DV(v)$ for each $v \in V(\mathcal{P})$.

\begin{observation}\label{observation2}
Let $G$ be the disjoint union of two graphs $G_1$ and $G_2$. Then
$\gamma(G)=\gamma(G_1)+ \gamma(G_2)$ and $\tau(G)=\tau(G_1) \cdot
\tau(G_2)$. For $v \in V(G_1)$, $DV_G(v)=DV_{G_1}(v) \cdot
\tau(G_2)$.
\end{observation}

\begin{proposition}\label{upperbound1}
For a fixed $v_0 \in V(G)$, we have
$$\tau(G) \le \sum_{v \in N[v_0]} DV_G(v) \le \tau(G) \cdot \gamma(G),$$ and both bounds are
sharp.
\end{proposition}

\begin{proof}
The upper bound follows from Observation \ref{observation}. For
the lower bound, note that every $\gamma(G)$-set $\Gamma$
must contain a vertex in $N[v_0]$: otherwise $\Gamma$ fails to dominate $v_0$.\\

For sharpness of the lower bound, take $v_0$ to be an end-vertex
of $P_{3k}$ for $k \ge 1$ (see Theorem \ref{theorem on paths} and
Corollary \ref{path on 3k}). For sharpness of the upper bound,
take as $v_0$ the central vertex of (A) in Figure \ref{figure1}. \hfill
\end{proof}

\begin{proposition} \label{upperbound2}
For any $v_0 \in V(G)$, $$\sum_{v \in N[v_0]} DV_G(v) \le \tau(G)
\cdot (1+ \deg_G(v_0)),$$ and the bound is sharp.
\end{proposition}

\begin{proof}
For each $v\in N[v_0]$, $DV_G(v)\leq \tau(G)$ and $|N[v_0]|=1+ \deg_G(v_0)$. Thus,\\
$$\sum_{v\in N[v_0]}DV(v) \leq \sum_{v\in N[v_0]}\tau(G)=\tau(G)\!\!\!\sum_{v\in N[v_0]}1=\tau(G)(1+ \deg_G(v_0)).$$\\
The upper bound is achieved for a graph of order $n$ for any
$n\geq 1$. Let $G_n$ be a graph on $n$ vertices containing an
isolated vertex. To see the sharpness of the upper bound, take as $v_0$
one of the isolates vertices, then the upper bound follows by
Observation \ref{observation2} and $\deg_G(v_0)=0$. \hfill
\end{proof}

\begin{figure}[htbp]
\begin{center}
\scalebox{0.5}{\input{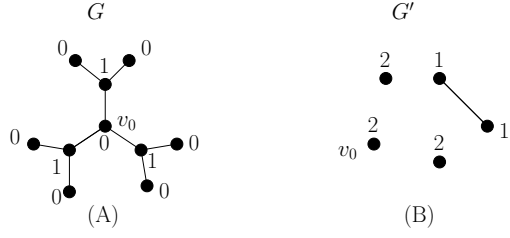}} \caption{Examples of
local domination values and their upper bounds}\label{figure1}
\end{center}
\end{figure}

We will compare two examples, where each example attains the upper
bound of Proposition \ref{upperbound1} or Proposition
\ref{upperbound2}, but not both. Let $v_0$ be the central vertex
of degree $3$, which is not a support vertex as in the graph (A)
of Figure \ref{figure1}. Then $\sum_{v \in N[v_0]} DV_G(v)=3$.
Note that $\tau (G) =1$, $\gamma (G)=3$, and $\deg_G(v_0)=3$.
Proposition \ref{upperbound1} yields the upper bound $\tau (G)
\cdot \gamma (G)=1\cdot 3=3$, which is sharp. But, the upper bound
provided by Proposition \ref{upperbound2}
is $\tau (G) \cdot (1+\deg_G(v_0))=1 \cdot (1+3)=4$, which is not sharp in this case.\\

Now, let $v_0$ be an isolated vertex as labeled in the graph (B)
of Figure \ref{figure1}. Then $\sum_{v \in N[v_0]}DV_{G'}(v)=2$.
Note that $\tau (G')=2$, $\gamma(G')=4$, and $\deg_{G'}(v_0)=0$.
Proposition \ref{upperbound2} yields the upper bound $\tau (G')
\cdot (1+ \deg_{G'}(v_0))=2 \cdot (1+0)=2$, which is sharp. But,
the upper bound provided by Proposition \ref{upperbound1} is
$\tau (G') \cdot \gamma (G')=2 \cdot 4=8$, which is not sharp in this case.\\

\begin{proposition}\label{subgraph-tau}
Let $H$ be a subgraph of $G$ with $V(H)=V(G)$. If
$\gamma(H)=\gamma(G)$, then $\tau(H) \le \tau(G)$.
\end{proposition}

\begin{proof}
By the first assumption, every DS for $H$ is a DS for $G$. By
$\gamma(H)=\gamma(G)$, it's guaranteed that every DS of minimum
cardinality for $H$ is also a DS \emph{of minimum cardinality for $G$}.
\end{proof}

The \emph{complement} $\overline{G}=(V(\overline{G}),
E(\overline{G}))$ of a graph $G$ is the graph such that
$V(\overline{G})=V(G)$ and $uv \in E(\overline{G})$ if and only if
$uv \not\in E(G)$. We recall the following

\begin{theorem}
Let $G$ be any graph of order $n$. Then
\begin{itemize}
\item[(i)] (\cite{Jaegar}, Jaegar and Payan)
$\gamma(G)+\gamma(\overline{G}) \le n+1$; and
\item[(ii)](\cite{B1}, p.304) $\gamma(G) \le n- \Delta(G)$.
\end{itemize}
\end{theorem}

\begin{proposition}
Let $G$ be a graph on $n=2m \ge 4$ vertices. If $G$ or
$\overline{G}$ is $mK_2$, then
$$DV_G(v)+DV_{\overline{G}}(v)=n-1+2^{\frac{n}{2}-1}.$$
\end{proposition}

\begin{proof}
Without loss of generality, assume $G=mK_2$ and label the vertices
of $G$ by $1,2,\ldots,2m$. Further assume that the vertex $2k-1$
is adjacent to the vertex $2k$, where $1 \le k \le m$. Then
$DV_G(1)=2^{m-1}$, which consists of the vertex 1 and one vertex
from each path $K_2$. By Observation \ref{observation1} and
Observation \ref{observation2}, $DV_G(v)=2^{m-1}=2^{\frac{n}{2}-1}$ for any $v\in V(G)$.\\

Now, consider $\overline{G}$ and the vertex labeled $1$ for ease
of notation. Since $\Delta(G)=n-2$, $\gamma(\overline{G})>1$.
Noting that $\{1, \alpha\}$ as $\alpha$ ranges from 2 to $2m$
enumerates all dominating sets of $\overline{G}$ containing the
vertex 1, we have $\gamma(\overline{G})=2$ and
$DV_{\overline{G}}(1)=n-1$. By Observation \ref{observation1},
$DV_{\overline{G}}(v)=n-1$ holds for any $v \in V(\overline{G})$.
Therefore, $DV_G(v)+DV_{\overline{G}}(v)=n-1+2^{\frac{n}{2}-1}$. \hfill
\end{proof}

Next we consider domination value of a graph $G$ when $\Delta(G)$ is
given.

\begin{observation} \label{degree n-1}
Let $G$ be a graph of order $n \ge 2$ such that $\Delta (G)=n-1$. Then
$\gamma(G)=1$ and $DV(v) \le 1$ for any $v \in V(G)$. Equality
holds if and only if $\deg_G(v)=n-1$.
\end{observation}

\begin{proposition}\label{degree n-2}
Let $G$ be a graph of order $n \ge 3$ such that $\Delta(G)=n-2$. Then
$\gamma(G)=2$ and $DV(v) \leq n-1$ for any $v \in V(G)$. Further,
if $\deg(v)=n-2$, then $DV(v)=|N[w]|$ where $vw \notin E(G)$.
\end{proposition}

\begin{proof}
Let $\deg_G(v)=\Delta(G)=n-2$, then $\gamma(G)>1$ and there's only
one vertex $w$ such that $vw \notin E(G)$. Clearly, $\{v, w\}$ is
a $\gamma(G)$-set; so $\gamma(G)=2$. Noticing that $v$ dominates
$N[v]$, we see that the number of $\gamma(G)$-sets containing $v$
is $|N[w]|$; i.e., $DV(v)=|N[w]| \le n-2$. \hfill
\end{proof}

\begin{theorem}\label{theorem on Delta(G)=n-3}
Let $G$ be a graph of order $n \ge 4$ and $\Delta (G)=n-3$. Fix a
vertex $v$ with $\deg_G(v)= \Delta(G)$.
\begin{itemize}
\item [(i)] If $G$ is disconnected, then $\gamma(G)=2$ with
$DV(v)=2$ or $\gamma(G)=3$ with $DV(v)\le n-3$. \item [(ii)] If
$G$ is connected, then $\gamma(G)=2$ with $DV(v) \le n-2$ or
$\gamma(G)=3$ with $DV(v) \le (\frac{n-1}{2})^2$.
\end{itemize}
\end{theorem}

\begin{proof}
Since $\deg_G(v)=\Delta(G)=n-3$, there are two vertices, say
$\alpha$ and $\beta$, such that $v \alpha, v \beta \not\in E(G)$.
We consider four cases.

\begin{figure}[htbp]
\begin{center}
\scalebox{0.5}{\input{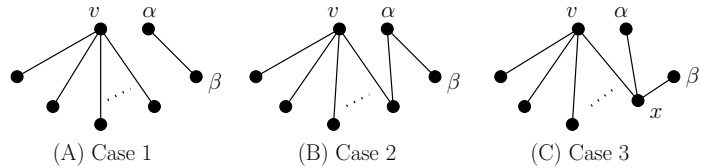}} \caption{Cases 1, 2, and
3 when $\Delta(G)=n-3$}\label{max1}
\end{center}
\end{figure}

\emph{Case 1. Neither $\alpha$ nor $\beta$ is adjacent to any
vertex in $N[v]$:} Let $G^{\prime}= \langle V(G) - \{\alpha,
\beta\} \rangle$. Then $\deg_{G'}(v)=n-3$ with $|V(G^{\prime})|=n-2$.
By Observation~\ref{degree n-1}, $\gamma(G')=1$ and
$DV_{G'}(v)=1$. First suppose $\alpha$ and $\beta$ are isolated
vertices in $G$. (Consider (A) of Figure \ref{max1} with the edge
$\alpha \beta$ being removed.) Observation~\ref{observation2},
together with $\gamma(\langle\{\alpha,\beta\}\rangle)=2$ and
$\tau(\langle\{\alpha,\beta\}\rangle)=1$, yields $\gamma(G)=3$ and
$DV_G(v)=1$. Next assume that $G$ has no isolated vertex, then $\alpha
\beta \in E(G)$ (see (A) of Figure \ref{max1}). Observation
\ref{observation2}, together with $\gamma(\langle\{\alpha,\beta \}\rangle)=1$ and
$\tau(\langle\{\alpha,\beta \}\rangle)=2$, yields $\gamma(G)=2$ and $DV_G(v)=2$.\\

\emph{Case 2. Exactly one of $\alpha$ and $\beta$ is adjacent to a
vertex in $N(v)$:} Without loss of generality, assume that
$\alpha$ is adjacent to a vertex in $N(v)$. First suppose that $G$
is not connected. Then $\alpha \beta \not\in E(G)$. (Consider (B)
of Figure \ref{max1} with the edge $\alpha \beta$ being removed.)
Let $G'=\langle V(G)-\{\beta\} \rangle$. Then $\deg_{G'}(v)=n-3$ with
$|V(G')|=n-1$. By Proposition \ref{degree n-2}, $\gamma(G')=2$ and
$DV_{G'}(v)=|N[\alpha]| \le n-3$. Observation \ref{observation2},
together with $\gamma(\langle \{\beta\} \rangle)=1$ and
$\tau(\langle \{\beta\}\rangle)=1$, yields $\gamma(G)=3$ and
$DV_G(v)=|N[\alpha]|\le n-3$. Next suppose that $G$ is connected.
Then $\alpha \beta \in E(G)$ and $\alpha$ is a support vertex of
$G$. (See (B) of Figure \ref{max1}.) Since $\Delta(G)<n-1$,
$\gamma(G)>1$. Since $\{v,\alpha\}$ is a $\gamma(G)$-set,
$\gamma(G)=2$. Noting that $v$ dominates $V(G)-\{\alpha, \beta\}$,
the number of $\gamma(G)$-sets containing $v$ equals the number of
vertices in $G$ that dominates both $\alpha$ and $\beta$. Thus $DV_G(v)=2$.\\

\emph{Case 3. There exists a vertex in $N(v)$, say $x$, that is
adjacent to both $\alpha$ and $\beta$:} Notice that $n \ge 6$ in
this case, since $vx, \alpha x, \beta x \in E(G)$ and
$\deg_G(v)=\Delta (G)$ (see (C) of Figure \ref{max1}). Since $\{v,
x\}$ is a $\gamma(G)$-set, $\gamma(G)=2$. If $\alpha\beta \not\in
E(G)$, then $DV(v)=|N[\alpha] \cap N[\beta]| \le n-3$. If
$\alpha\beta \in E(G)$, then $|N[\alpha] \cap N[\beta]| \le n-4$
since $\Delta (G)=n-3$. Noting both $\{v, \alpha\}$ and $\{v,
\beta\}$ are $\gamma(G)$-sets, we have $DV(v)=2+|N[\alpha] \cap N[\beta]| \le n-2$.\\

\begin{figure}[htbp]
\begin{center}
\scalebox{0.5}{\input{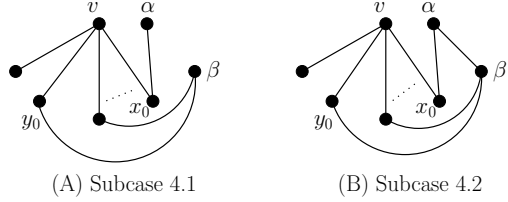}} \caption{Subcases 4.1
and 4.2 when $\Delta(G)=n-3$}\label{max2}
\end{center}
\end{figure}

\emph{Case 4. There exist vertices in $N(v)$ that are adjacent to
$\alpha$ and $\beta$, but no vertex in $N(v)$ is adjacent to both
$\alpha$ and $\beta$:} Let $x_0 \in N(v) \cap N(\alpha)$ and $y_0
\in N(v) \cap N(\beta)$. We consider two
subcases.\\

\emph{Subcase 4.1. $\alpha \beta \not\in E(G)$ (see (A) of Figure
\ref{max2}):} First, assume $\gamma(G)=2$. This is possible when
$\{x_0, y_0\}$ is a $\gamma(G)$-set satisfying $N[x_0] \cup
N[y_0]=V(G)$. Notice that there's no $\gamma(G)$-set containing
$v$ when $\gamma(G)=2$ since there's no vertex in $G$ that is
adjacent to both $\alpha$ and $\beta$. Thus $DV(v)=0$. Second,
assume $\gamma(G)>2$. Since $\{v, \alpha, \beta\}$ is a
$\gamma(G)$-set, $\gamma(G) =3$. Noticing that every
$\gamma(G)$-set contains a vertex in $N[\alpha]$ and a vertex in
$N[\beta]$ and that $N[\alpha] \cap N[\beta] =\emptyset$, we see
$$DV(v) = |N[\alpha]| \cdot |N[\beta]| \le
\left(\frac{|N[\alpha]|+|N[\beta]|}{2}\right)^2 \le
\left(\frac{n-1}{2}\right)^2,$$ where the first inequality is the
arithmetic-geometric mean inequality (i.e., $\frac{a+b}{2} \geq
\sqrt{ab}$ for $a, b \geq 0$).\\

\emph{Subcase 4.2. $\alpha \beta \in E(G)$ (see (B) of Figure
\ref{max2}):} Since $\{v, \alpha\}$ is a $\gamma(G)$-set,
$\gamma(G)=2$. Since there's no vertex in $N(v)$ that is adjacent
to both $\alpha$ and $\beta$, there are only two $\gamma(G)$-sets
containing $v$, i.e., $\{v, \alpha\}$ and $\{v, \beta\}$. Thus
$DV(v)=2$. \hfill
\end{proof}

\noindent\textbf{Remark.} In the proof of Theorem~\ref{theorem on
Delta(G)=n-3}, we observe that one may have $DV(v)=0$ even though
$\deg_G(v)=\Delta(G) \le n-3$. See Figure \ref{zeroTDV-max-deg}
for a graph of order $n$, $\deg_G(v)= \Delta(G)=n-3$,
$\gamma(G)=2$, and $DV(v)=0$.

\begin{figure}[htbp]
\begin{center}
\scalebox{0.5}{\input{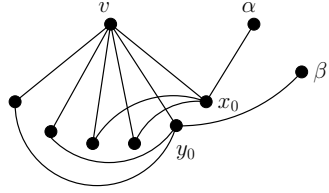}} \caption{A graph of
order $9$, $\deg_G(v)=\Delta(G)=6$, $DV(v)=0$ with a unique
$\gamma$-set $\{x_0, y_0\}$}\label{zeroTDV-max-deg}
\end{center}
\end{figure}


\section{Domination value on complete $n$-partite graphs}

For a \emph{complete} $n$-partite graph $G$, let $V(G)$ be
partitioned into $n$-partite sets $V_1$, $V_2$, $\ldots$, $V_n$,
and let $a_i=|V_i|\geq 1$ for each $1\leq i \leq n$, where $n\geq
2$.

\begin{proposition}
Let $G=K_{a_1,a_2, \ldots, a_n}$ be a complete $n$-partite graph
with $a_i \ge 2$ for each $i$ ($1 \le i \le n$). Then
$$\tau(G)=\frac{1}{2}\left[\left(\sum_{i=1}^{n} a_i\right)^2 - \sum_{i=1}^{n} a_i^2\right] \hskip .2in
\mbox{ and } \hskip .2in DV(v)=\left(\sum_{i=1}^{n} a_i\right)-a_j
\mbox{ if } v \in V_j.$$
\end{proposition}

\begin{proof}
Since $\Delta(G) < |V(G)|-1$, $\gamma(G)>1$. Any two vertices from
different partite sets form a $\gamma(G)$-set, so $\gamma(G)=2$.
If $v \in V_j$, then
\begin{equation}\label{DV for n-partite}
DV(v)=\deg_G(v)=\left(\sum_{i=1}^{n} a_i\right)-a_j.
\end{equation}
From Observation \ref{observation} and (\ref{DV for n-partite}),
we have
\begin{eqnarray*}
\sum_{j=1}^{n}\sum_{v\in V_j} DV(v) =2 \tau(G) \
\Longleftrightarrow \ \sum_{j=1}^{n}\left(a_j\sum_{i=1}^{n} a_i
-a_{j}^{2}\right) =2\tau(G) \\ 
 \Longleftrightarrow \ \
\left(\sum_{i=1}^{n} a_i\right)\left(\sum_{j=1}^{n}
a_j\right)-\sum_{j=1}^{n} a_j^2 =2\tau(G),
\end{eqnarray*}
and thus the formula for $\tau(G)$ follows.\hfill
\end{proof}

\begin{proposition}\label{n-partite2}
Let $G=K_{a_1,a_2, \ldots, a_n}$ be a complete $n$-partite graph
such that $a_i=1$ for some $i$, say $a_j=1$ for $j=1,2,\ldots, k$,
where $1 \le k \le n$. Then $\tau(G)=k$ and
\begin{equation*} DV(v)= \left\{
\begin{array}{ll}
1 & \mbox{ if } v \in V_j \  (1 \le j \le \ k)\\
0 & \mbox{ if } v \in V_j \ (k+1 \le j \le n) .
\end{array} \right.
\end{equation*}
\end{proposition}

\begin{proof}
Since $\Delta(G)=|V(G)|-1$, by Observation \ref{degree n-1},
$\gamma(G)=1$ and $DV(v)$ follows. By Observation
\ref{observation}, together with $\gamma(G)=1$, we have $\tau(G)=\sum_{v
\in V(G)}DV_G(v)=k$.\hfill
\end{proof}

If $a_i=1$ for each $i$ ($1\leq i\leq n$), then $G=K_n$. As an
immediate consequence of Proposition \ref{n-partite2}, we have the
following.

\begin{corollary}
If $G=K_n$ ($n \ge 1$), then $\tau(G)= n$ and $DV(v)=1$ for each
$v \in V(K_n)$.
\end{corollary}

If $n=2$, then $G=K_{a_1, a_2}$ is a complete bi-partite graph.
\begin{corollary}
If $G=K_{a_1, a_2}$, then
\begin{equation*}
\tau(G)= \left\{
\begin{array}{ll}
a_1 \cdot a_2 & \mbox{ if } a_1, a_2 \ge 2 \\
2 & \mbox{ if } a_1=a_2=1\\
1 & \mbox{ if } \{a_1, a_2\} = \{1,x\}, \mbox{ where }x>1 .
\end{array} \right.
\end{equation*}
If $a_1, a_2 \ge 2$, then
\begin{equation*}
DV(v)= \left\{
\begin{array}{ll}
a_2 & \mbox{ if } v \in V_1 \\
a_1 & \mbox{ if } v \in V_2 .
\end{array} \right.
\end{equation*}
If $a_1=a_2=1$, $DV(v)=1$ for any $v$ in $K_{1,1}$. 
If $\{a_1, a_2\} = \{1,x\}$ with $x>1$, say $a_1=1$ and $a_2=x$,
then
\begin{equation*} DV(v)= \left\{
\begin{array}{ll}
1 & \mbox{ if } v \in V_1 \\
0 & \mbox{ if } v \in V_2 .
\end{array} \right.
\end{equation*}
\end{corollary}


\section{Domination value on cycles}

Let the vertices of the cycle $C_n$ be labeled $1$ through $n$
consecutively in counter-clockwise order, where $n \ge 3$. Observe
that the domination value is constant on the vertices of $C_n$,
for each $n$, by vertex-transitivity. Recall that
$\gamma(C_n)=\lceil \frac{n}{3} \rceil$ for $n \ge 3$ (see
p.364, \cite{CZ}).\\

\noindent {\bf Examples.} (a) $DM(C_4)=\{\{1,
2\}, \{1,3\}, \{1,4\}, \{2, 3\}, \{2, 4\}, \{3, 4\}\}$ since $\gamma(C_4)=2$;
so $\tau(C_4)=6$ and $DV(i)=3$ for each $i \in V(C_4)$.\\

(b) $\gamma(C_6)=2$, $DM(C_6)=\{\{1, 4\}, \{2, 5\}, \{3, 6\} \}$;
so $\tau(C_6)=3$ and $DV(i)=1$ for each $i \in V(C_6)$.\\

\begin{theorem}\label{theorem on cycles}
For $n \ge 3$,
\begin{equation*}
\tau(C_n) = \left\{
\begin{array}{lr}
\ 3 & \mbox{ if } n \equiv 0  \mbox{ (mod 3)} \ \\
\ n(1+\frac{1}{2} \lfloor \frac{n}{3} \rfloor) & \mbox{ if } n \equiv 1 \mbox{ (mod 3)} \ \\
\ n & \mbox{ if } n \equiv 2  \mbox{ (mod 3)} .
\end{array} \right.
\end{equation*}
\end{theorem}

\begin{proof}
First, let $n=3k$, where $k \ge 1$. Here $\gamma(C_{n})=k$; a
$\gamma(C_{n})$-set $\Gamma$ comprises $k$ $K_1$'s and $\Gamma$ is
fixed by the choice of the first $K_1$. There exists exactly one
$\gamma(C_{n})$-set containing the vertex $1$, and there are two
$\gamma(C_{n})$-sets omitting the vertex 1 such as $\Gamma$
containing
the vertex 2 and $\Gamma$ containing the vertex $n$. Thus $\tau(C_{n})=3$.\\

Second, let $n=3k+1$, where $k \ge 1$. Here $\gamma(C_{n})=k+1$; a
$\gamma(C_{n})$-set $\Gamma$ is constituted in exactly one of the
following two ways: 1) $\Gamma$ comprises $(k-1)$ $K_1$'s and one
$K_2$; 2) $\Gamma$ comprises $(k+1)$ $K_1$'s. \\

\emph{Case 1) $\langle \Gamma \rangle \cong (k-1)K_1 \cup K_2$:} Note that
$\Gamma$ is fixed by the choice of the single $K_2$. Choosing a
$K_2$ is the same as choosing its initial vertex
in the counter-clockwise order. Thus $\tau=3k+1$.\\

\emph{Case 2) $\langle \Gamma \rangle \cong (k+1)K_1$:} Note that, since each
$K_1$ dominates three vertices, there are exactly two vertices,
say $x$ and $y$, each of whom is adjacent to two distinct $K_1$'s
in $\Gamma$. And $\Gamma$ is fixed by the placements of $x$ and
$y$. There are $n=3k+1$ ways of choosing $x$. Consider the
$P_{3k-2}$ (a sequence of $3k-2$ slots) obtained as a result of
cutting from $C_{n}$ the $P_3$ centered about $x$. Vertex $y$ may
be placed in the first slot of any of the
$\lceil\frac{3k-2}{3}\rceil=k$ subintervals of the $P_{3k-2}$. As
the order of selecting the two vertices
$x$ and $y$ is immaterial, $\tau=\frac{(3k+1)k}{2}$. \\

Summing over the two disjoint cases, we get
$$\tau(C_{n})=(3k+1)+\frac{(3k+1)k}{2}=(3k+1)\left(1+\frac{k}{2}\right)= n \left(1+\frac{1}{2} \left\lfloor\frac{n}{3} \right\rfloor \right).$$

Finally, let $n=3k+2$, where $k \ge 1$. Here $\gamma(C_{n})=k+1$;
a $\gamma(C_{n})$-set $\Gamma$ comprises of only $K_1$'s and is
fixed by the placement of the only vertex which is adjacent to two
distinct $K_1$'s in $\Gamma$. Thus $\tau(C_{n})=n$. \hfill
\end{proof}

\begin{corollary}
Let $v\in V(C_n)$, where $n \ge 3$. Then
\begin{equation*}
DV(v) = \left\{
\begin{array}{lr}
\ 1 & \mbox{ if } n \equiv 0  \mbox{ (mod 3)} \ \\
\ \frac{1}{2} \lceil\frac{n}{3}\rceil (1+ \lceil\frac{n}{3}\rceil) & \mbox{ if } n \equiv 1  \mbox{ (mod 3)} \ \\
\ \lceil\frac{n}{3}\rceil & \mbox{ if } n \equiv 2 \mbox{ (mod 3) }.
\end{array} \right.
\end{equation*}
\end{corollary}

\begin{proof}
It follows by Observation~\ref{observation}, Observation~\ref{observation1}, and Theorem~\ref{theorem on cycles}.~\hfill
\end{proof}


\section{Domination value on paths}

Let the vertices of the path $P_n$ be labeled $1$ through $n$
consecutively. Recall that $\gamma(P_n)=\lceil \frac{n}{3} \rceil$ for $n \ge 2$.\\

\noindent {\bf Examples.} (a) $\gamma(P_4)=2$, $DM(P_4)=\{
\{1,3\}, \{1,4\}, \{2, 3\}, \{2,4\}\}$; so $\tau(P_4)=4$ and
$DV(i)=2$ for each $i \in V(P_4)$.\\

(b) $\gamma_t(P_5)=2$, $DM(P_5)=\{\{1, 4\}, \{2,4\}, \{2,5\}\}$;
so $\tau(P_5)=3$, and
\begin{equation*}
DV(i)= \left\{
\begin{array}{ll}
1 & \mbox{ if } i =1,5\\
2 & \mbox{ if } i=2,4\\
0 & \mbox{ if } i =3 .
\end{array} \right.
\end{equation*}

\noindent\textbf{Remark.} Since $P_n \subset C_n$ with the same vertex set,
by Proposition \ref{subgraph-tau}, we have $\tau(P_n) \le
\tau(C_n)$ for $n \ge 3$, as one can verify from the theorem
below.

\begin{theorem}\label{theorem on paths}
For $n \ge 2$,
\begin{equation*} \tau(P_n) = \left\{
\begin{array}{lr}
\ 1 & \mbox{ if } n \equiv 0  \mbox{ (mod 3)} \ \\
\ n+ \frac{1}{2}\lfloor \frac{n}{3} \rfloor (\lfloor \frac{n}{3}\rfloor-1) & \mbox{ if } n \equiv 1  \mbox{ (mod 3)} \ \\
\ 2+ \lfloor \frac{n}{3} \rfloor & \mbox{ if } n \equiv 2 \mbox{
(mod 3)} .
\end{array} \right.
\end{equation*}
\end{theorem}

\begin{proof}
First, let $n=3k$, where $k \ge 1$. Then $\gamma(P_n)=k$ and a
$\gamma(P_n)$-set $\Gamma$ comprises $k$ $K_1$'s. In this case,
each vertex in $\Gamma$ dominates three vertices, and no vertex of
$P_{n}$ is dominated by more than one vertex. Thus none of the
end-vertices of $P_n$ belongs to any $\Gamma$, which contains and
is fixed by the vertex $2$; hence $\tau(P_n)=1$.\\

Second, let $n=3k+1$, where $k \ge 1$. Here $\gamma(P_n)=k+1$; a
$\gamma(P_n)$-set $\Gamma$ is constituted in exactly one of the
following two ways: 1) $\Gamma$ comprises $(k-1)$ $K_1$'s and one
$K_2$; 2) $\Gamma$ comprises $(k+1)$ $K_1$'s.\\

\emph{Case 1) $\langle \Gamma \rangle \cong (k-1)K_1 \cup K_2$, where $k \ge
1$:} Note that $\Gamma$ is fixed by the placement of the single
$K_2$, and none of the end-vertices belong to any $\Gamma$, as
each component with cardinality $c$ in $\langle \Gamma \rangle$ dominates $c+2$
vertices. Initial vertex of $K_2$ may be placed in one of the $n
\equiv 2$ (mod $3$) slots. Thus $\tau=k$.\\

\emph{Case 2) $\langle \Gamma \rangle \cong (k+1)K_1$, where $k \ge 1$:} A
$\Gamma$ containing both end-vertices of the path is unique (no
vertex is doubly dominated). The number of $\Gamma$ containing
exactly one of the end-vertices (one doubly dominated vertex) is
$2{k \choose 1}=2k$. The number of $\Gamma$ containing none of the
end-vertices (two doubly dominated vertices) is ${k \choose 2}=
\frac{k(k-1)}{2}$. Thus
$\tau=1+2k+\frac{k(k-1)}{2}$.\\

Summing over the two disjoint cases, we get
$$\tau(P_n)=k+\left(1+2k+\frac{k(k-1)}{2} \right)=3k+1+ \frac{k(k-1)}{2}=
n+\frac{1}{2} \left\lfloor\frac{n}{3}\right\rfloor
\left(\left\lfloor \frac{n}{3}\right\rfloor-1 \right).$$

Finally, let $n=3k+2$, where $k \ge 0$. Here $\gamma(P_n)=k+1$,
and $\gamma(P_n)$-set $\Gamma$ comprises of ($k+1$) $K_1$'s. Note
that there's no $\Gamma$ containing both end-vertices of $P_n$.
The number of $\Gamma$ containing exactly one of the end-vertices
(no doubly dominated vertex) of the path is two.  The number of
$\Gamma$ containing neither of the end-vertices (one doubly
dominated vertex) is $k$. Summing the two disjoint cases, we have
$\tau(P_n)=2+k=2+ \lfloor \frac{n}{3}\rfloor$. \hfill
\end{proof}

For the domination value of a vertex on $P_n$, note that
$DV(v)=DV(n+1-v)$ for $1 \le v \le n$ as $P_n$ admits the obvious
automorphism carrying $v$ to $n+1-v$. More precisely, we have the
classification result which follows. First, as an immediate
consequence of Theorem \ref{theorem on paths}, we have the
following result.

\begin{corollary} \label{path on 3k}
Let $v \in V(P_{3k})$, where $k \ge 1$. Then
\begin{equation*}
DV(v)= \left\{
\begin{array}{ll}
0 & \mbox{ if } v \equiv 0,1  \mbox{ (mod 3)} \\
1 & \mbox{ if } v \equiv 2  \mbox{ (mod 3)} .
\end{array} \right.
\end{equation*}
\end{corollary}

\begin{proposition}
Let $v \in V(P_{3k+1})$, where $k \ge 1$. Write $v=3q+r$, where
$q$ and $r$ are non-negative integers such that $0 \le r < 3$.
Then, noting $\tau(P_{3k+1})=\frac{1}{2}(k^2+5k+2)$, we have
\begin{equation*} DV(v)= \left\{
\begin{array}{ll}
\frac{1}{2}q(q+3) & \mbox{ if } v \equiv 0 \mbox{ (mod 3) } \ \\
(q+1)(k-q+1) & \mbox{ if } v \equiv 1 \mbox{ (mod 3) } \ \\
\frac{1}{2}(k-q)(k-q+3) & \mbox{ if } v \equiv 2 \mbox{ (mod 3) }.
\end{array} \right.
\end{equation*}
\end{proposition}

\begin{proof}
Let $\Gamma$ be a $\gamma(P_{3k+1})$-set for $k \ge 1$. We
consider two cases.\\

\textit{Case 1) $\langle \Gamma \rangle \cong (k-1)K_1 \cup K_2$, where $k \ge
1$:} Denote by $DV^1(v)$ the number of such $\Gamma$'s containing
$v$. Noting $\tau=k$ in this case, we have
\begin{equation}\label{path 3k+1 case1}
DV^1(v)= \left\{
\begin{array}{ll}
q & \mbox{ if } v \equiv 0  \mbox{ (mod 3)} \ \\
0 & \mbox{ if } v \equiv 1 \mbox{ (mod 3)} \ \\
k-q & \mbox{ if } v \equiv 2 \mbox{ (mod 3)} .
\end{array} \right.
\end{equation}
We prove by induction on $k$. One can easily check (\ref{path 3k+1
case1}) for $k=1$. Assume that (\ref{path 3k+1 case1}) holds for
$G=P_{3k+1}$ and consider $G'=P_{3k+4}$. First, notice that each
$\Gamma$ of the $k$ $\gamma(P_{3k+1})$-sets of $G$ induces a
$\gamma(P_{3k+4})$-set $\Gamma'=\Gamma \cup \{3k+3\}$ of $G'$.
Additionally, $G'$ has the $\gamma(P_{3k+4})$-set $\Gamma^*$ that
contains and is determined by $\{3k+2, 3k+3\}$, which does not
come from any $\gamma(P_{3k+1})$-set of $G$. The presence of
$\Gamma^*$ implies that $DV^1_{G'}(v)=DV^1_G(v)+1$ for $v \equiv
2$ (mod 3), where $v \le 3k+1$. Clearly, $DV^1_{G'}(3k+2)=1$,
$DV^1_{G'}(3k+3)=k+1$, and $DV^1_{G'}(3k+4)=0$.\\

\textit{Case 2) $\langle \Gamma \rangle \cong (k+1)K_1$, where $k \ge 1$:}
Denote by $DV^2(v)$ the number of such $\Gamma$'s containing $v$.
First, suppose both end-vertices belong to the unique $\Gamma$ and
denote by $DV^{2,1}(v)$ the number of such $\Gamma$'s containing
$v$. Then we have
\begin{equation}\label{Case 3-1}
DV^{2,1}(v)= \left\{
\begin{array}{ll}
0 & \mbox{ if } v \equiv 0, 2  \mbox{ (mod 3)} \ \\
1 & \mbox{ if } v \equiv 1 \mbox{ (mod 3)} .
\end{array} \right.
\end{equation}

Second, suppose exactly one end-vertex belongs to each $\Gamma$;
denote by $DV^{2,2}(v)$ the number of such $\Gamma$'s containing
$v$. Then, noting $\tau=2k$ in this case, we have
\begin{equation}\label{Case 3-2}
DV^{2,2}(v)= \left\{
\begin{array}{ll}
q & \mbox{ if } v \equiv 0 \mbox{ (mod 3)} \ \\
k & \mbox{ if } v \equiv 1 \mbox{ (mod 3)} \ \\
k-q & \mbox{ if } v \equiv 2 \mbox{ (mod 3)} .
\end{array} \right.
\end{equation}
We prove by induction on $k$. One can easily check (\ref{Case
3-2}) for $k=1$. Assume that (\ref{Case 3-2}) holds for
$G=P_{3k+1}$ and consider $G'=P_{3k+4}$. First, notice that each
$\Gamma$ of the $k$ $\gamma(P_{3k+1})$-sets of $G$ containing the
left end-vertex $1$ induces a $\gamma(P_{3k+4})$-set
$\Gamma'=\Gamma \cup \{3k+3\}$ of $G'$. Second, each $\Gamma$ of
$k$ $\gamma(P_{3k+1})$-sets of $G$ containing the right end-vertex
$3k+1$ induces a $\gamma(P_{3k+4})$-set $\Gamma'=\Gamma \cup
\{3k+4\}$ of $G'$. Third, a $\gamma(P_{3k+1})$-set $\Gamma$ of $G$
containing $1$ and $3k+1$ (both left and right end-vertices of
$G$) induces a $\gamma(P_{3k+4})$-set $\Gamma^{*1}=\Gamma \cup
\{3k+3\}$ of $G'$ (making $3k+2$ the only doubly dominated vertex
in $G'$). Additionally, $\Gamma^{*2}=\{v \in V(P_{3k+1}) \mid v \equiv 2 \mbox{ (mod 3)}\} \cup \{3k+2, 3k+4\}$ is a
$\gamma(P_{3k+4})$-set for $G'$, which does not come from any
$\gamma(P_{3k+1})$-set of $G$. The presence of $\Gamma^{*1}$ and
$\Gamma^{*2}$ imply that
\begin{equation*}
DV^{2,2}_{G'}(v)= \left\{
\begin{array}{ll}
DV^{2,2}_G(v) & \mbox{ if } v \equiv 0 \mbox{ (mod 3)} \ \\
DV^{2,2}_G(v)+1 & \mbox{ if } v \equiv 1,2 \mbox{ (mod 3)}
\end{array} \right.
\end{equation*}
for $v \le 3k+1$.
Clearly, $DV^{2,2}_{G'}(3k+2)=1$, $DV^{2,2}_{G'}(3k+3)=k+1$, and $DV^{2,2}_{G'}(3k+4)=k+1$.\\

Third, suppose no end-vertex belongs to $\Gamma$; denote by
$DV^{2,3}(v)$ the number of such $\Gamma$'s containing $v$. Then,
noting $\tau= {k\choose 2}$ in this case and setting ${a\choose
b}=0$ when $a<b$, we have
\begin{equation} \label{Case 3-3}
DV^{2,3}(v)= \left\{
\begin{array}{ll}
\frac{1}{2}(q-1)q & \mbox{ if } v \equiv 0 \mbox{ (mod 3)} \ \\
q(k-q) & \mbox{ if } v \equiv 1 \mbox{ (mod 3)} \ \\
\frac{1}{2}(k-q-1)(k-q) & \mbox{ if } v \equiv 2\mbox{ (mod 3)} .
\end{array} \right.
\end{equation}
Again, we prove by induction on $k$. Since $DV^{2,3}(v)=0$ for
each $v \in V(P_{4})$, we consider $k \ge 2$. One can easily check
(\ref{Case 3-3}) for the base, $k=2$. Assume that (\ref{Case 3-3})
holds for $G=P_{3k+1}$ and consider $G'=P_{3k+4}$, where $k \ge
2$. First, notice that each $\Gamma$ of the $k\choose 2$
$\gamma(P_{3k+1})$-sets of $G$ containing neither end-vertices of
$G$ induces a $\gamma(P_{3k+4})$-set $\Gamma'=\Gamma \cup
\{3k+3\}$ of $G'$. Additionally, each $\Gamma_r$ of the $k$
$\gamma(P_{3k+1})$-sets of $G$ containing the right-end vertex
$3k+1$ of $G$ induces a $\gamma(P_{3k+4})$-set $\Gamma_r'=\Gamma_r
\cup \{3k+3\}$ of $G'$ (making $3k+2$ one of the two
doubly-dominated vertices in $G'$): If we denote by $DV^r_G(v)$
the number of such $\Gamma_r$'s containing $v$ in $G$, then one
can readily check
\begin{equation*}
DV^r_G(v)= \left\{
\begin{array}{ll}
0 & \mbox{ if } v \equiv 0 \mbox{ (mod 3)} \ \\
q & \mbox{ if } v \equiv 1 \mbox{ (mod 3)} \ \\
k-q & \mbox{ if } v \equiv 2\mbox{ (mod 3)} ,
\end{array} \right.
\end{equation*}
again by induction on $k$. Thus, the presence of $\Gamma'_r$
implies $DV_{G'}^{2,3}(v)=DV_G^{2,3}(v)+DV^r_G(v)$ for $v \le
3k+1$. Clearly, $DV^{2,3}_{G'}(3k+2)=0=DV^{2,3}_{G'}(3k+4)$
and $DV^{2,3}_{G'}(3k+3)= {k\choose 2}+k=\frac{1}{2}k(k+1)$.\\

Summing over the three disjoint cases (\ref{Case 3-1}), (\ref{Case
3-2}), and (\ref{Case 3-3}) for $\langle \Gamma \rangle \cong (k+1)K_1$, we have
\begin{equation}\label{path 3k+1 case3}
DV^2(v)= \left\{
\begin{array}{ll}
q + \frac{1}{2}(q-1)q& \mbox{ if } v \equiv 0 \mbox{ (mod 3)} \ \\
1+k+q(k-q) & \mbox{ if } v \equiv 1 \mbox{ (mod 3)} \ \\
k-q+ \frac{1}{2}(k-q-1)(k-q)& \mbox{ if } v \equiv 2 \mbox{ (mod
3)} .
\end{array} \right.
\end{equation}
Now, by summing over (\ref{path 3k+1 case1}) and (\ref{path 3k+1
case3}), i.e., $DV(v)=DV^1(v)+DV^2(v)$, we obtain the formula
claimed in this proposition.~\hfill
\end{proof}

\begin{proposition}
Let $v \in V(P_{3k+2})$, where $k \ge 0$. Write $v=3q+r$, where
$q$ and $r$ are non-negative integers such that $0 \le r < 3$.
Then, noting $\tau(P_{3k+2})=k+2$, we have
\begin{equation*}
DV(v)= \left\{
\begin{array}{ll}
0 & \mbox{ if } v \equiv 0 \mbox{ (mod 3) } \ \\
1+q & \mbox{ if } v \equiv 1 \mbox{ (mod 3) } \ \\
k+1-q & \mbox{ if } v \equiv 2 \mbox{ (mod 3) } .
\end{array} \right.
\end{equation*}
\end{proposition}

\begin{proof}
Let $\Gamma$ be a $\gamma(P_{3k+2})$-set for $k \ge 0$. Then $\langle
\Gamma \rangle \cong (k+1) K_1$. Note that no $\Gamma$ contains both
end-vertices of $P_{3k+2}$.\\

First, suppose $\Gamma$ contains
exactly one end-vertex, and denote by $DV'(v)$ the number of such
$\Gamma$'s containing $v$. Noting $\tau=2$ in this case, for $v
\in V(P_{3k+2})$, we have
\begin{equation}\label{path on 3k+2 case1}
DV'(v)= \left\{
\begin{array}{ll}
0 & \mbox{ if } v \equiv 0 \mbox{ (mod 3) } \ \\
1 & \mbox{ if } v \equiv 1,2 \mbox{ (mod 3) }.
\end{array} \right.
\end{equation}

Next, suppose $\Gamma$ contains no end-vertices (thus $k \ge 1$),
and denote by $DV''(v)$ the number of such $\Gamma$'s containing
$v$. Noting $\tau=k$ in this case, we have
\begin{equation}\label{path on 3k+2 case2}
DV''(v)= \left\{
\begin{array}{ll}
0 & \mbox{ if } v \equiv 0 \mbox{ (mod 3) } \ \\
q & \mbox{ if } v \equiv 1 \mbox{ (mod 3) } \ \\
k-q & \mbox{ if } v \equiv 2 \mbox{ (mod 3) } .
\end{array} \right.
\end{equation}
We prove by induction on $k$. One can easily check (\ref{path on
3k+2 case2}) for the base, $k=1$. Assume that (\ref{path on 3k+2
case2}) holds for $G=P_{3k+2}$ and consider $G'=P_{3k+5}$. First,
notice that each $\Gamma$ of the $k$ $\gamma(P_{3k+2})$-sets
containing neither end-vertex of $G$ induces a
$\gamma(P_{3k+5})$-set $\Gamma'=\Gamma \cup \{3k+4\}$.
Additionally, the only $\gamma(P_{3k+2})$-set $\Gamma$ of $G$
containing the right-end vertex $3k+2$ of $G$ induces a
$\gamma(P_{3k+5})$-set $\Gamma^\star=\Gamma \cup \{3k+4\}$ of $G'$
(making $3k+3$ the only doubly-dominated vertex). The presence of
$\Gamma^\star$ implies that
\begin{equation*}
DV_{G'}''(v)= \left\{
\begin{array}{ll}
DV_G''(v) & \mbox{ if } v \equiv 0, 1 \mbox{ (mod 3) } \\
DV_G''(v)+1 & \mbox{ if } v \equiv 2 \mbox{ (mod 3) }
\end{array} \right.
\end{equation*}
for $v \le 3k+2$. Clearly, $DV_{G'}''(3k+3)=0=DV_{G'}''(3k+5)$ and
$DV_{G'}''(3k+4)=k+1$.\\

Now, by summing over the two disjoint cases (\ref{path on 3k+2
case1}) and (\ref{path on 3k+2 case2}), i.e.,
$DV(v)=DV'(v)+DV''(v)$, we obtain the formula claimed in this
proposition. \hfill
\end{proof}

\textit{Acknowledgement.} The author thanks Cong X. Kang for
suggesting the concept of \emph{domination value} and his valuable
comments and suggestions. The author also thanks the referee for a couple of helpful comments.

\nocite{*}
\bibliographystyle{amsplain}
\bibliography{}

\end{document}